\documentclass[acmsmall]{acmart}
\usepackage[linesnumbered,ruled,vlined]{algorithm2e}

\usepackage[T1]{fontenc}     

\def\BibTeX{{\rm B\kern-.05em{\sc i\kern-.025em b}\kern-.08emT\kern-.1667em\lower.7ex\hbox{E}\kern-.125emX}}

\setcopyright{acmcopyright}
\acmJournal{TOMS}
\acmYear{2019}\acmVolume{0}\acmNumber{0}\acmArticle{0}\acmMonth{0}
\acmDOI{0}

\usepackage{graphicx}

\begin{document}

%
\title[An incremental descent method for multi-objective optimization]{An incremental descent method for multi-objective optimization}

%
\author{I. F. D. Oliveira}
\email{ivodavid@gmail.com}
\affiliation{%
  \institution{Institute of Science, Engineering and Technology, Federal University of the Valleys of Jequitinhonha and Mucuri}
  \streetaddress{R. Cruzeiro, 1}
  \city{Te\'ofilo Otoni}
  \state{Minas Gerais}
  \country{Brazil}
  \postcode{39803-371}
}
\author{R. H. C. Takahashi}
\email{taka@mat.ufmg.br}
\affiliation{%
  \institution{Department of Mathematics, Federal University of Minas Gerais}
  \streetaddress{Av. Pres. Ant\^onio Carlos, 6627}
  \city{Belo Horizonte}
  \state{Minas Gerais}
  \country{Brazil}
  \postcode{31270-901}
}

%
\renewcommand{\shortauthors}{Oliveira and Takahashi}

%
\begin{abstract}
 Current state-of-the-art multi-objective optimization solvers, by computing gradients of all $m$ objective functions per iteration, produce after $k$ iterations a measure of proximity to critical conditions that is upper-bounded by $O(1/\sqrt{k})$ when the objective functions are assumed to have $L-$Lipschitz continuous gradients; i.e. they require $O(m/\epsilon^2)$  gradient and function computations to produce a measure of proximity to critical conditions bellow some target $\epsilon$. We reduce this to $O(1/\epsilon^2)$ with a method that requires only a constant number of gradient and function computations per iteration; and thus, we obtain for the first time a multi-objective descent-type method with a query complexity cost that is unaffected by increasing values of $m$. For this, a brand new multi-objective descent direction is identified, which we name the \emph{central descent direction}, and, an incremental approach is proposed. Robustness properties of the central descent direction are established, measures of proximity to critical conditions are derived, and, the incremental strategy for finding solutions to the multi-objective problem is shown to attain convergence properties unattained by previous methods. To the best of our knowledge, this is the first method to achieve this with no additional a-priori information on the structure of the problem, such as done by scalarizing techniques, and, with no pre-known information on the regularity of the objective functions other than Lipschitz continuity of the gradients.
\end{abstract}

%
%


%

%

%
\maketitle

\section{Introduction}

In this paper we deal with the problem of finding \emph{critical} points of multi-objective optimization problems. The multi-objective function $F:\mathbb{R}^n\to\mathbb{R}^m$ is defined as the concatenation of $m$ mono-objective functions as $F(\boldsymbol{x})\equiv [f_1(\boldsymbol{x}), ..., f_m(\boldsymbol{x})]^T$ that are assumed to have L-Lipschitz continuous gradients, i.e. for some $L\geq 0$ and for all $\boldsymbol{x},\boldsymbol{y} \in \mathbb{R}$ and all $i = 1,...,m$ the inequality  $||\nabla f_i(\boldsymbol{x})-\nabla f_i(\boldsymbol{y})||\leq L||\boldsymbol{x}-\boldsymbol{y}||$ holds. And, following the definition of \cite{fliege1}, a point $\boldsymbol{x}\in \mathbb{R}^n$ is said to be \emph{critical} (also referred to as stationary) if and only if there does not exist a direction $\boldsymbol{v}\in \mathbb{R}^n$ such that $\nabla f_i(\boldsymbol{x})^T \boldsymbol{v}< 0$ for all $i = 1,..., m$. 
 
The search for stationary points is motivated by the fact that standard first order conditions for local optimality, akin to $\nabla f(\boldsymbol{x}) = 0$ in mono-objective optimization, is that $\boldsymbol{x}$ be stationary \cite{luenberger}. Descent direction algorithms, analogous to those employed in mono-objective problems, have been devised and analysed extensively for the multi-objective setting, including: (i) \emph{steepest descent} strategies \cite{fliege1,vieira,fliege3}; (ii) \emph{Newton} type methods \cite{fliege2,povalej}; (iii) \emph{projected gradient} strategies \cite{drummond,cruz}; amongst others \cite{perez,gebken,moudden,tanabe}. These generally follow the same structure as mono-objective optimization strategies:

\begin{algorithm}[H]
\SetAlgoLined
initialize $\boldsymbol{\hat{x}}\in \mathbb{R}^n$\; 
 \While{True}{
  compute a descent direction $\boldsymbol{d}$\;
  find an appropriate step-size $\alpha \in \mathbb{R}_+$\;
  update the estimate  $\boldsymbol{\hat{x}} \leftarrow \boldsymbol{\hat{x}} + \alpha \boldsymbol{d}$\;
 }
 \caption{\label{alg:descent_direction}Descent direction method}
\end{algorithm}

The multi-objective steepest descent method is perhaps the benchmark to which other methods are typically compared \cite{fliege1,fliege2,fliege3,fukuda,cocchi}. For any given estimate $\boldsymbol{\hat{x}}$, the steepest descent method defines $\boldsymbol{d}$ as $\boldsymbol{d}\equiv \boldsymbol{V}_s(\boldsymbol{\hat{x}})$ where
\begin{equation} \label{eq:steepest_descent}
\boldsymbol{V}_s(\hat{x}) \equiv \text{arg} \min_{\boldsymbol{V}\in\mathbb{R}^n} \max_{i = 1,...,m}\nabla f_i(\boldsymbol{\hat{x}})^T \boldsymbol{V}  +\frac{1}{2}||\boldsymbol{V}||^2. 
\end{equation}
And, if line 3 of Algorithm \ref{alg:descent_direction} is implemented with $\boldsymbol{d}\equiv \boldsymbol{V}_s(\boldsymbol{\hat{x}})$, and  line 4 is implemented with step-size $\alpha$ produced by Armijo's backtracking procedure delineated in \cite{fliege1} or the Fibonacci search of \cite{vieira}, then, under mild technical assumptions, global convergence is guaranteed irrespective of the starting point, analogous to the guarantees of mono-objective gradient descent. See Theorem 1 of \cite{fliege1} and Theorem 3 of \cite{vieira} for the details. Furthermore, the computational complexity of the mono-objective gradient descent method, which upperbounds the norm of the gradient by $O(1/\sqrt{k})$ after $k$ steps, is also echoed by it's multi-objective counterpart as the analogous measure of proximity to critical conditions $\min_{l = 1,...,k} ||\boldsymbol{V}_s(\hat{\boldsymbol{x}}^l)||$ is also upper-bounded by  $O(1/\sqrt{k})$ after $k$ steps under the Lipschitz continuous gradient assumption; see Theorem 3.1 of \cite{fliege3}. This worst case query complexity, in the mono-objective setting, is known to be optimal and is not improved on even with the addition of second order information, such as with the use of Newton's method; see Section 3 of \cite{cartis}. Analogous convergence guarantees to those of gradient descent in mono-objective optimization are also known for convex functions as well as strongly convex functions when the multi-objective steepest descent method is employed \cite{fliege3,fukuda}. And furthermore, similar results have been developed for Newton-type methods as well as projected gradient-type methods in the multi-objective optimization literature \cite{fliege2,povalej,drummond,cruz}.

\paragraph{Problem description} While multi-objective steepest descent recovers the $O(1/\sqrt{k})$ rate of convergence of mono-objective problem solving, since  the computation of $\boldsymbol{d}$ in each iteration requires querying all $m$ gradients, and, the production of $\alpha$ in both exact and inexact line searches requires querying all $m$ functions per iteration, the overall query complexity of multi-objective steepest descent is still $m$ times that of mono-objective optimization; i.e. when we take into consideration the dependence on $m$ the number of calls to gradient and function evaluations to upper-bound the measure of proximity to critical conditions by $\epsilon$ is more precisely of the order of $ O(m/\epsilon^2)$. That is, with the same query complexity budget, gradient descent can produce solutions to each of the $m$ mono-objective minimization problems separately with the similar $\epsilon$ worst case bounds as one run of the multi-objective steepest descent method. A close analysis of this complexity gap produces a somewhat puzzling set of informal propositions:\\
\\
\textbf{P1.} \textit{The requirement `$\tilde{x}$ is stationary with respect to $F(x) \equiv [f_1(x),..., f_m(x)]^T$' is a relaxation of the mono-objective equivalent `$\tilde{x}$ is stationary with respect to $F(x) \equiv [f_1(x)]$'.}\\
\textbf{P2.} \textit{Multi-objective steepest descent with $m$ objective functions collects $m$ times more queries of gradient and function values per iteration than gradient descent applied to the mono-objective problem.}\\
\textbf{P3.} \textit{Multi-objective steepest descent requires the same amount of iterations as gradient descent to produce an estimate with similar worst case bounds.}\\ 

That is, combining \textbf{P1} to \textbf{P3} we find that multi-objective steepest descent requires more queries to produce one estimate  that is associated with weaker requirements than the $m$ mono-objective solutions produced by gradient descent with the same query complexity budget. One should expect quite the opposite: to produce an estimate `$\boldsymbol{\tilde{x}}$ stationary with respect to $F(\boldsymbol{x}) \equiv [f_1(\boldsymbol{x}),..., f_m(\boldsymbol{x})]^T$  (for some pre-specified tolerance)' should be less computationally demanding than to produce an estimate `$\boldsymbol{\tilde{x}}$ stationary with respect to $F(\boldsymbol{x}) \equiv [f_1(\boldsymbol{x})]^T$ (for some pre-specified tolerance)'. In fact, a trivial proof that this is the case is obtained by \textbf{P1} alone: since critical points with respect to each mono-objective function are also critical with respect to the multi-objective problem, the computational cost of gradient descent on any one objective function is an upper-bound on the cost of the task of producing `$\boldsymbol{\tilde{x}}$ stationary with respect to $F(\boldsymbol{x}) \equiv [f_1(\boldsymbol{x}),..., f_m(\boldsymbol{x})]^T$ (for some pre-specified tolerance)'.

To the best of our knowledge, and somewhat unsatisfyingly, the strategy of ``ignore all but one objective and employ an off-the-shelf mono-objective optimization algorithm'' is still more computationally efficient than the methods available in the multi-objective optimization literature when the task at hand is to produce \emph{any} multi-objective stationary point. This `trivial' approach of ignoring all but one objective, however unsatisfactory, can be seen as an instance of the well known \emph{scalarization} techniques often employed in multi-objective optimization; the most basic of which involves: (i) selecting a weight vector $\boldsymbol{\pi}\in \mathbb{R}^m_+$ such that $\pi_1+...+\pi_m = 1$; then (ii) defining a scalarized mono-objective function $F_{\boldsymbol{\pi}}(\boldsymbol{x}) \equiv \sum_{i=1}^m \pi_if_i(\boldsymbol{x})$; and finally (iii) finding a solution to the mono-objective minimization problem $\boldsymbol{\hat{x}_{\pi}} = \text{arg}\min F_{\boldsymbol{\pi}}(\boldsymbol{x})$ via gradient descent (or any other mono-objective optimization method of choice). These techniques are further discussed ahead, however, we point out that overall, scalarization choices not only (almost always) stumble on the same query complexity problems pointed above, but also, have a well documented collection of a priori shortcomings associated with the use of surrogate loss functions\footnote{Information is either lost or added when surrogate optimization strategies are employed thus placing the practitioner at risk of producing artificial non-realizable solutions, or, at risk of  producing sub-optimal solutions, or, to the very least at risk of producing an artificially narrowed down (meta-parameter dependent) set of alternatives to the decision maker; decision making alternatives that become difficult to quantify and compared to the hypothetical alternatives that would be produced by the original problem.} instead of dealing with the original problem directly. Hence, we consider two questions to address the complexity gap that arises when considering \textbf{P1} to \textbf{P3}:\\
\\
\textbf{Q1.} \textit{Is it possible to produce global convergence guarantees to stationary points of $F(\boldsymbol{x}) \equiv [f_1(\boldsymbol{x}), ..., f_m(\boldsymbol{x})]$ with a multi-objective approach that has an iteration cost no greater than that of mono-objective iterative approaches?}\\
\textbf{Q2.} \textit{Is it possible to obtain an analogous $\epsilon$-measure of proximity to critical conditions of $F(\boldsymbol{x}) \equiv [f_1(\boldsymbol{x}), ..., f_m(\boldsymbol{x})]$ with only $O(1/\epsilon^2)$ queries to gradients and function values using a multi-objective approach without arbitrary scalarization? That is, a global computational cost that is unaffected by increasing values of m?}\\

In this paper we attempt to answer both \textbf{Q1} and \textbf{Q2} separately: \textbf{Q1} is addressed in Theorem \ref{the:incremental_central_descent} and \textbf{Q2} in Theorem \ref{the:lowest_incremental_central_descent}. Both results provide positive answers to the questions above, and, to the best of our knowledge, Algorithms \ref{alg:central_basic} and \ref{alg:lowest_central_descent}, analysed in the aforementioned theorems, provide for the first time efficient first order descent direction methods for multi-objective optimization. That is, unlike the descent direction methods proposed so far within the multiobjective literature, our methods are the first to have a query complexity that is unaffected by increasing values of $m$ (a property that can substantially reduce the computational cost of the search the higher the value of $m$).

The methods we propose here are perhaps more similar to what are known as \emph{incremental gradient} methods \cite{gurbuzbalaban,bertsekas2,nedic} rather than \emph{steepest descent methods} within the mono-objective literature. Incremental gradient methods are alternatives to full gradient methods developed to reduce the computational cost of solving mono-objective minimization problems of the form $\min F_{\boldsymbol{\pi}}(\boldsymbol{x}) \equiv \sum_{i=1}^m \pi_if_i(\boldsymbol{x})$ by estimating the descent directions with only small number of gradients $\nabla f_i(\cdot)$ per iteration. When done concomitantly with randomization of the sampled gradients, then, these are some times referred to as stochastic gradient descent methods \cite{vaswani,ruder}.   Much of the literature of incremental gradient methods can be directly transposed to the multi-objective setting by considering scalarizing techniques (discussed ahead). However, as mentioned above scalarizing with a-priori weights either requires a-priori knowledge of the proper scaling of the objectives or may produce undesirable solutions associated with the use of surrogate loss functions, and, despite some analogies between the  techniques here developed and the incremental gradient strategies, our method deals with the multi-objective problem directly thus avoiding such problems all together. Furthermore, to the best of our understanding, to fully enjoy the theoretical speed-ups of incremental gradient techniques known so far,  strong convexity of the loss function, or to the very least convexity, must be assumed, whereas here we only assume Lipschitz continuity of the gradient.

\paragraph{Paper  layout} In the remainder of this section we give a brief overview of scalarization techniques employed to reduce the multi-objective problem to a mono-objective problem as well as a brief overview of incremental methods. There we point out how mono-objective incremental strategies are already equipped to tackle scalarized problems with a reduced query complexity when compared to the state-of-the-art in multi-objective optimization, albeit, with the shortcomings associated with the use of surrogate loss functions instead of tackling the multi-objective problem directly. In Section \ref{sec:direction} we define the \emph{central descent direction}, an alternative generalization of the mono-objective steepest descent direction much similar to $\boldsymbol{V}_s$ in (\ref{eq:steepest_descent}) that plays a key role in our analysis. There we provide key properties of the central descent direction  that separate it from $\boldsymbol{V}_s$ as well as descent directions derived from scalarization. In this section we also provide geometric and robustness properties of the central descent direction as well as several illustrations to facilitate the understanding of the methods developed in the following sections. The main results are provided in Section \ref{sec:incremental_descent}. In the first part of Section \ref{sec:incremental_descent} we define the \emph{incremental central descent} algorithm and provide  the most general global convergence guarantees under the assumption of non-additive vanishing step-sizes  irrespective of the starting point; producing either (i) a sub-sequence of estimates with converging critical conditions or, (ii) a sequence of function values that tends to $-\infty$. Then, in second part we provide one specific instantiation of the incremental central descent method and we prove that, for functions bounded from below, the critical conditions converge at the rate of $1/\sqrt{k}$ irrespective of the starting point. The algorithm analysed in the first part of Section  \ref{sec:incremental_descent} has an iteration cost of at most one gradient computation per iteration, and, in the second part at most two gradient computations per iteration plus a mono-objective Armijo-type line search thus  improving the current best known complexity guarantees in multi-objective optimization by eliminating the dependence on $m$.

\subsection{On scalarizing techniques and incremental approaches}

Several multi-objective alternatives to descent direction algorithms can be  found in the literature,  including multi-objective evolutionary approaches \cite{veldhuizen,zitzler},  multi-objective trust region methods \cite{ryu,villacorta,thomann} amongst others \cite{mahdavi-amiri,tsionas}. However, perhaps the most popular approach in dealing with multi-objective optimization problems is characterized by using mono-objective strategies to scalarized instances of the multi-objective problem \cite{ulivieri,ghane-kanafi,gunantara,eichfelder,kasimbeyli}. These have the immediate benefit of inheriting the convergence guarantees associated with the mono-objective optimization algorithm, and, can typically be inserted within larger iterative processes that enable the construction of a representative sample of efficient solutions of the multi-objective problem - a feature that is often the motivating factor behind adopting a multi-objective formulation in decision making processes \cite{cocchi}. The most basic of these scalarizing approaches involves: (i) selecting a weight vector $\boldsymbol{\pi}\in \mathbb{R}^m_+$ such that $\pi_1+...+\pi_m = 1$; then (ii) defining a scalarized mono-objective function $F_{\boldsymbol{\pi}}(\boldsymbol{x}) \equiv \sum_{i=1}^m \pi_if_i(\boldsymbol{x})$; and (iii) finding a solution to the mono-objective minimization problem $\boldsymbol{\hat{x}_{\pi}} = \text{arg}\min F_{\boldsymbol{\pi}}(\boldsymbol{x})$ via gradient descent (or any other mono-objective optimization method of choice). Under this approach, the computation of the direction $\boldsymbol{d}$ is typically much simpler than (\ref{eq:steepest_descent}) as the gradient of $F_{\boldsymbol{\pi}}(\cdot)$ is simply the weighted sum of the individual gradients $\nabla F_{\boldsymbol{\pi}}(\boldsymbol{x}) \equiv \sum_{i=1}^m \pi_i\nabla f_i(\boldsymbol{x})$. More sophisticated scalarizing methods include constrained optimization formulations in the form ``$\min f_i(\boldsymbol{x})$ subject to $f_j(\boldsymbol{x})\leq z_j$ for all $j\neq i$'' for varying values of $\boldsymbol{z} \in \mathbb{R}^m$ as well as unconstrained goal-programming formulations in the form $\min ||F(\boldsymbol{x}) - \boldsymbol{z}||$ for some pre-specified norm $||\cdot||$ in $\mathbb{R}^m$ and varying values of $\boldsymbol{z}$; see \cite{ulivieri} for an extensive list.

Since, scalarization employs off the shelf mono-objective methods to solve an instance of the multi-objective problem, these can also attain (trivially) the minmax $O(1/\sqrt{k})$ rate of convergence to find an efficient solution to the original problem. However, these formulations also (almost always) require $m$ gradient computations per iteration and thus they carry the same overall worst-case query complexity as the multi-objective steepest descent approach in producing a stationary point; i.e. when the dependence on $m$ is made explicit, the query complexity is more precisely of $O(m/\sqrt{k})$. For example, the goal programming approach of reducing the multi-objective problem by minimizing $g(\boldsymbol{x
})=||F(\boldsymbol{x}) - \boldsymbol{z}||_2$ still requires the computation of all $m$ gradients  to calculate a descent direction of the surrogate loss function just as the weighted sum approach of minimizing $F_{\boldsymbol{\pi}}(\boldsymbol{x}) \equiv \sum_{i=1}^m \pi_if_i(\boldsymbol{x})$. And, the constrained optimization formulations, when solved with a projected descent direction type approach will typically require more than one computation of each individual function $f_i(\cdot)$ for verifying feasibility, which in turn adds dependency on $m$ in each step even if the gradient is computed on one function alone during the iteration. Thus, to the best of our knowledge, the dependence on $m$ does not hold only under two exceptions: the first being when the scalarization parameters trivially discard objectives (such as by taking many null weights $\pi_i$ in the minimization of $F_{\boldsymbol{\pi}}(\boldsymbol{x})\equiv \sum_{i = 1}^m \pi_if_i(\boldsymbol{x})$\ ) and the second being when the mono-objective solver is capable of exploiting the inherit structure of a scalarized problem (such as with the use of incremental gradient descent approaches). 

\paragraph{Incremental gradient descent} The incremental gradient descent is a mono-objective optimization technique developed to reduce the query complexity of minimization problems with objective functions of the form $F_{\boldsymbol{\pi}}(\boldsymbol{x}) \equiv \sum_{i=1}^m \pi_if_i(\boldsymbol{x})$. These techniques, together with stochastic gradient descent methods, have served as the back-bone of practical solvers developed for extremely large and computationally expensive neural network training problems \cite{vaswani,ruder,gurbuzbalaban,bertsekas2,nedic}. The main difference between both stochastic and deterministic incremental strategies  and full gradient methods is that incremental strategies do not compute the full gradient of $F_{\boldsymbol{\pi}}(\cdot)$ in each iteration; instead, in each iteration $k$ a small sample of gradients $\nabla f_i(\boldsymbol{x}_k)$ is collected, typically of size $O(1)$, and a descent direction $\boldsymbol{d}$ is estimated/updated with both new and old information available. The choice of the sampled gradients and the method of producing $\boldsymbol{d}$ is where the main difference between different incremental methods lie.

Perhaps the first deterministic results that achieved a speed of convergence independent of $m$ are found in \cite{blatt}, where the iterations are defined as
\begin{equation}\label{eq:increment_grad}
\boldsymbol{x}^{k+1} \equiv \boldsymbol{x}^k - \frac{\alpha}{L} \sum_{l = 0}^{L-1} \nabla f_{(k-l)_L}(\boldsymbol{x}^{k-l}); \end{equation}
for some sufficiently small constant $\alpha\in \mathbb{R}_+$ and some $L\in \mathbb{N}$. There they find that for quadratic functions an estimate $\hat{\boldsymbol{x}}$ that satisfies $||F_{\boldsymbol{\pi}}(\hat{\boldsymbol{x}})-F_{\boldsymbol{\pi}}(\boldsymbol{x}^*)||\leq \epsilon$ can be produced in less than $O(log(1/\epsilon))$ iterations. Since this seminal result, variations of (\ref{eq:increment_grad}) have been produced that extended these guarantees for general classes of Lipschitz continuous and strongly convex functions \cite{liu,vanli,gurbuzbalaban2,tseng,mokhtari}; achieving guarantees that are even tighter than full gradient descent methods \cite{mokhtari}.

The application of incremental descent methods to \emph{scalarized} versions of the multi-objective problem are fairly straight forward: it suffices to (i) choose weights; (ii) reduce the multi-objective problem to a mono-objective problem; and, (iii) employ an off-the-shelf incremental solver to the scalarized problem. And, given that the objective functions satisfy the convexity requirements of  \cite{liu,vanli,gurbuzbalaban2,tseng,mokhtari}, then a reduced computational complexity that does not depend on $m$ can be enjoyed. However, as mentioned above, all the drawbacks associated with the reduction of a multi-objective to a mono-objective problem are also to be expected. If such a reduction to a mono-objective problem were known a-priori, then there would be no point in adopting a multi-objective formulation in the first place; multi-objective formulations are typically chosen precisely because no such satisfiable reductionistic approach is known. Furthermore, the convexity assumptions necessary for guaranteed improvements on the convergence speed can be hard to ensure when multiple objectives (typically of very different natures) are considered. Unlike neural network training problems where each $f_i$ is typically the loss functions measured on one training point, in multi-objective optimization each $f_i$ typically a measures objectives that are impossible to put on one common scale: one may represent cost of production, the other safety protocol measures, and the last variance in quality control etc.

\section{The central descent direction} \label{sec:direction}

For non-critical points $\boldsymbol{x}$, any $\boldsymbol{v}$ that satisfies $\nabla f_i(\boldsymbol{x})^T \boldsymbol{v}<0$ for all $i = 1,..., m$ is called a descent direction.  In this paper one specific uniquely defined descent direction, which we call the \emph{central descent direction}, is of interest. The central descent direction is denoted by  $\boldsymbol{V}_c(\boldsymbol{x})$ and is defined as:
\begin{equation}\label{eq:central_descent}
\boldsymbol{V}_c(\boldsymbol{x})\equiv \left\{\begin{array}{ll}
\text{argmin } & \tfrac{1}{2}||\boldsymbol{V}||^2 \\
\text{st. } & \nabla f_i(\boldsymbol{x})^T \boldsymbol{V} \leq -||\nabla f_i(\boldsymbol{x})|| \text{ for all i = 1, ..., m}\end{array} \right.
\end{equation}
Problem (\ref{eq:central_descent}) can only admit a solution if $\boldsymbol{x}$ is non-critical or if at least one of the gradients $\nabla f_i(\boldsymbol{x})$ is null. Otherwise, critical points $\boldsymbol{x}$ with non-null gradients $\nabla f_i(\boldsymbol{x})$ do not admit a solution. 
Hence, the conditions for a point in $\mathbb{R}^n$ to be critical can be restated in terms of the central descent direction: 
\begin{lemma} 
A point $\boldsymbol{x}\in \mathbb{R}^n$ is critical if and only if (i) at least one gradient $\nabla f_i(\boldsymbol{x})$ is null for some $i= 1, ..., m$ or (ii) the central descent direction $\boldsymbol{V}_c(\boldsymbol{x})$ defined in (\ref{eq:central_descent}) is empty or (iii) both.
\end{lemma}
\begin{proof}
Follows immediately from the definition of critical points.
\end{proof}
Furthermore, the central descent direction provides a means of ``measuring'' the proximity to critical conditions (as discussed bellow). To see this we will need the following lemma:
\begin{lemma}\label{lem:diverge}
Let $\{\boldsymbol{x}_k\}_{k = 1,...,\infty}$ be a sequence of non-critical points converging to $\boldsymbol{x}_{\infty}\in \mathbb{R}^n$. Then, if neither the gradients $\nabla f_i(\boldsymbol{x}_k)$ goes to zero with $k\to \infty$, we have that either (ii) the accumulation point $\boldsymbol{x}_{\infty}$ is critical and  $||\boldsymbol{V}_c(\boldsymbol{x}_k)|| \to \infty$; or (ii) the accumulation point  $\boldsymbol{x}_{\infty}$ is not critical and $\boldsymbol{V}_c(\boldsymbol{x}_k)$ converges to $\boldsymbol{V}_c(\boldsymbol{x}_{\infty})$.
\end{lemma}
\begin{proof}
This proof will use make use of continuity-type results on the dependence of the central descent direction with respect to the gradients, described in Lemma \ref{lem:lower} whose proof is found in the appendix Section \ref{sec:aux_lem1}.

\begin{lemma}\label{lem:lower}
Given a collection of vectors $\boldsymbol{g}_{i}$ for $i = 1, ...,m$  define $\mathcal{S} \equiv \{\boldsymbol{v} \text{ st. } \boldsymbol{g}_{i}^T \boldsymbol{v} \leq -||\boldsymbol{g}_{i}|| \text{ for all i = 1, ...,m } \}$. Then:
\begin{enumerate}
\item If for some $R>0$ the intersection of $\{ \boldsymbol{v}\text{ st. } ||\boldsymbol{v}||\leq R\}$ with $\mathcal{S}$ is empty, then, there exists an $\epsilon>0$ such that for every collection of $\hat{\boldsymbol{g}}_i$'s such that $||\hat{\boldsymbol{g}}_i - \boldsymbol{g}_i||\leq \epsilon$ for all $i = 1,...,m$ the intersection of $\{ \boldsymbol{v}\text{ st. } ||\boldsymbol{v}||\leq R\}$ with $\{\boldsymbol{v} | \hat{\boldsymbol{g}}_{i}^T \boldsymbol{v} \leq -||\hat{\boldsymbol{g}}_{i}|| \text{ for all i = 1, ...,m } \}$ is also empty.
\item If $\boldsymbol{v}_I$ is vector in $\mathbb{R}^n$ that satisfies $ \boldsymbol{g}_{i}^T \boldsymbol{v}_I < -||\boldsymbol{g}_{i}|| \text{ for all i = 1, ...,m }$, then, there exists an $\epsilon>0$ such that for every collection of $\hat{\boldsymbol{g}}_i$'s in $\mathbb{R}^n$ such that $||\hat{\boldsymbol{g}}_i - \boldsymbol{g}_i||\leq \epsilon$ for all $i = 1,...,m$ the vector $\boldsymbol{v}_I$ also satisfies $\hat{\boldsymbol{g}}_{i}^T \boldsymbol{v}_I < -||\hat{\boldsymbol{g}}_{i}|| \text{ for all i = 1, ...,m }$.
\item If $\boldsymbol{v}_E$ is a vector in $\mathbb{R}^n$ that satisfies $ \boldsymbol{g}_{i}^T \boldsymbol{v}_E > -||\boldsymbol{g}_{i}|| \text{ for some i = 1, ...,m }$, then, there exists an $\epsilon>0$ such that for every collection of $\hat{\boldsymbol{g}}_i$'s in $\mathbb{R}^n$ such that $||\boldsymbol{g}_i - \hat{\boldsymbol{g}}_i||\leq \epsilon$ for all $i = 1,...,m$  the vector $\boldsymbol{v}_E$ also satisfies $\hat{\boldsymbol{g}}_{i}^T \boldsymbol{v}_E > -||\hat{\boldsymbol{g}}_{i}|| \text{ for some i = 1, ...,m }$.
\end{enumerate}
\end{lemma}
\begin{proof}
Refer to the Appendix Section \ref{sec:aux_lem1}.
\end{proof}

Given that the gradients are L-Lipschitz continuous, we must have that $||\nabla f_i(\boldsymbol{x}_k) - \nabla f_i(\boldsymbol{x}_{\infty})|| \to 0$ as $k\to \infty$. Let us assume that $\boldsymbol{x}_{\infty}$ is critical with $\nabla f_i(\boldsymbol{x}_{\infty}) \neq 0$ for all i = 1, ..., m. 
Under these conditions the feasible set  $\mathcal{S}_{\infty} \equiv \{\boldsymbol{v} \text{ st. } \nabla f_i(\boldsymbol{x}_{\infty})^T \boldsymbol{v} \leq - ||\nabla f_i(\boldsymbol{x}_{\infty})|| \text{ for all i = 1, ..., m } \}$ is empty, and thus for any given $R>0$ the intersection between $\{\boldsymbol{v} \text{ st. }||\boldsymbol{v}||\leq R\}$ and $\mathcal{S}_{\infty}$ is empty. Hence, as a consequence of Lemma \ref{lem:lower} part 1 we have: for any $R>0$ there will exist an $\epsilon>0$ such that any collection of $\hat{\boldsymbol{g}}_i$'s in which $||\hat{\boldsymbol{g}}_i-\nabla f_i (\boldsymbol{x}_{\infty})||\leq \epsilon$ for all $i = 1,...,m$ and any $\boldsymbol{v}$  which satisfies $||\boldsymbol{v}||\leq R$ we have that the intersection between $\{\boldsymbol{v} \text{ st. }||\boldsymbol{v}||\leq R\}$ and $\{\boldsymbol{v} \text{ st. }\hat{\boldsymbol{g}}_i^T\boldsymbol{v}\leq -||\hat{\boldsymbol{g}}_i ||\text{ for all i = 1,...,m }\}$ is also empty. As a consequence of this fact,  as $\nabla f_i(\boldsymbol{x}_{k})$ approaches $\nabla f_i(\boldsymbol{x}_{\infty})$, the value of $||\boldsymbol{V}_c(\boldsymbol{x}_k)||$ must arbitrarily increase.

Now, let us assume that $\boldsymbol{x}_{\infty}$ is non-critical, and hence, $\boldsymbol{V}_c(\boldsymbol{x}_{\infty})$ is non empty. We recall the definition of the feasible set $\mathcal{S}_{\infty}$ as $\mathcal{S}_{\infty} \equiv \{\boldsymbol{v} \text{ st. } \nabla f_i(\boldsymbol{x}_{\infty})^T \boldsymbol{v} \leq - ||\nabla f_i(\boldsymbol{x}_{\infty})|| \text{ for all i = 1, ..., m } \}$, and, we will refer to the \emph{interior} of $\mathcal{S}_{\infty}$ as the set of vectors that satisfy all the inequalities strictly, i.e. the interior of $\mathcal{S}_{\infty}$ is the set $\{\boldsymbol{v} \text{ st. } \nabla f_i(\boldsymbol{x}_{\infty})^T \boldsymbol{v} < - ||\nabla f_i(\boldsymbol{x}_{\infty})|| \text{ for all i = 1, ..., m } \}$. By considering the vector $\boldsymbol{v}_{\delta}\equiv (1+\delta) \boldsymbol{V}_c(\boldsymbol{x}_{\infty})$ for any $\delta>0$ it is easy to see that  $\boldsymbol{v}_{\delta}$ is in the interior of $\mathcal{S}_{\infty}$. As a consequence of Lemma \ref{lem:lower} part 2, for every $\boldsymbol{v}_I$ in the interior of $\mathcal{S}_{\infty}$ there exists an $\epsilon>0$ such that for every collection of $\hat{\boldsymbol{g}}_i$'s for $i=1,...,m$ that satisfy $||\hat{\boldsymbol{g}}_i - \nabla f_i(\boldsymbol{x}_{\infty})|| \leq \epsilon$ we will have $\boldsymbol{v}_I$ also in the interior of $\{\boldsymbol{v} \text{ st. } \hat{\boldsymbol{g}}_i^T \boldsymbol{v} \leq -||\hat{\boldsymbol{g}}_i|| \text{ for all i = 1, ..., m }\}$. The same can be said about points in the exterior of $\mathcal{S}_{\infty}$ by applying Lemma \ref{lem:lower} part 3, where $\boldsymbol{v}_E$ is said to be in the exterior of $\mathcal{S}_{\infty}$ if there exists an $i$ between $1$ and $m$ such that $\nabla f_i(\boldsymbol{x}_{\infty})^T \boldsymbol{v}_E>-||\nabla f_i(\boldsymbol{x}_{\infty})||$. Hence, if any subsequence of $\boldsymbol{V}_k$'s converges, then, since it cannot converge to the interior of $\mathcal{S}_{\infty}$ nor the exterior, it must converge to the subset of $\mathcal{S}_{\infty}$ where at least one of the constraints $\nabla f_i(\boldsymbol{x}_{\infty})^T\boldsymbol{v} \leq -||\nabla f_i(\boldsymbol{x}_{\infty})||$ is satisfied with an equality.

Now, notice that every point in the interior of $\mathcal{S}_k\equiv \{\boldsymbol{v} \text{ st. }  \nabla f_i(\boldsymbol{x}_{k})^T \boldsymbol{v} < - ||\nabla f_i(\boldsymbol{x}_{k})|| \text{ for all i = 1, ..., m } \}$ provides an upper-bound on the value of $||\boldsymbol{V}_c(\boldsymbol{x}_k)||$ since $\boldsymbol{V}_c(\boldsymbol{x}_k)$ is the minimizer of $||\boldsymbol{v}||$ over $\mathcal{S}_k$. One such upper-bound is sufficient to recognize that $\boldsymbol{V}_c(\boldsymbol{x}_k)$ remains bounded for $k>\bar{k}$ for some $\bar{k}\in \mathbb{N}$. Thus, since bounded sequences must have at least one converging sub-sequence, all that is left is to show that every such sub-sequence converges to $\boldsymbol{V}_c(\boldsymbol{x}_{\infty})$. 

Since every point in the interior of $\mathcal{S}_k$ provides an upper-bound on the value of $||\boldsymbol{V}_c(\boldsymbol{x}_k)||$ we know that $\lim_{k\to \infty}||\boldsymbol{V}_c(\boldsymbol{x}_k)|| \leq \lim_{\delta \to 0}||\boldsymbol{v}_{\delta}||=  ||\boldsymbol{V}_c(\boldsymbol{x}_{\infty})||$. Now, to see that $\lim_{k\to \infty}||\boldsymbol{V}_c(\boldsymbol{x}_k)||$ also satisfies $\lim_{k\to \infty}||\boldsymbol{V}_c(\boldsymbol{x}_k)|| \geq  ||\boldsymbol{V}_c(\boldsymbol{x}_{\infty})||$ notice that if $\{\boldsymbol{v}\text{ st. }||\boldsymbol{v}||\leq R\}\cap\mathcal{S}_{\infty}$ is empty for some $R>0$, then, by Lemma \ref{lem:lower} part 1 we have that for $\epsilon>0$ sufficiently small, any collection of $\hat{\boldsymbol{g}}_i$'s that satisfy $||\hat{\boldsymbol{g}}_i-\nabla f_i(\boldsymbol{x}_{\infty})||\leq \epsilon$ will also have an empty intersection between $\{\boldsymbol{v}\text{ st. }||\boldsymbol{v}||\leq R\}$ and $\{ \boldsymbol{v} \text{ st. } \hat{\boldsymbol{g}}_i^T\boldsymbol{v} \leq -||\hat{\boldsymbol{g}}_i \text{ for all i = 1, ..., m}|| \}$. This is the case for any $R<||\boldsymbol{V}_c(\boldsymbol{x}_{\infty})||$, and thus for any $R<||\boldsymbol{V}_c(\boldsymbol{x}_{\infty})||$ and for any $k\geq \bar{k}$ for some sufficiently large $\bar{k}\in \mathbb{N}$ the value of $||\boldsymbol{V}_c(\boldsymbol{x}_{\infty})||$ is greater than or equal to $R$. Taking the limit of $R$ to $||\boldsymbol{V}_c(\boldsymbol{x}_{\infty})||$ we obtain that $\lim_{k\to\infty}||\boldsymbol{V}_c(\boldsymbol{x}_{k})|| \geq ||\boldsymbol{V}_c(\boldsymbol{x}_{\infty})||$.

Defining $\boldsymbol{v}_{\text{lim}} \equiv \lim_{k\to \infty} \boldsymbol{V}_c(\boldsymbol{x}_k)$,  since we have established that  $||\boldsymbol{v}_{\text{lim}}|| = ||\boldsymbol{V}_c(\boldsymbol{x}_{\infty})||$ while satisfying $\nabla f_i (\boldsymbol{x}_{\infty})^T\boldsymbol{v}_{\text{lim}}\leq -||\nabla f_i (\boldsymbol{x}_{\infty})||$  for all i = 1, ..., m (because it cannot be in the interior or the exterior of $\mathcal{S}_{\infty}$); and, since $\boldsymbol{V}_c(\boldsymbol{x}_{\infty})$ is the unique minimizer of $||\boldsymbol{v}||$ over $\mathcal{S}_{\infty}$, we must conclude that $\lim_{k\to \infty}\boldsymbol{V}_c(\boldsymbol{x}_k) = \boldsymbol{V}_c(\boldsymbol{x}_{\infty})$.  
\end{proof}

Thus, Lemma \ref{lem:diverge} gives us a measure of proximity to critical conditions of multi-objective optimization problems in terms of the central descent direction. In mono-objective optimization, convergence to critical conditions are typically measured in terms of the norm of the gradient; that is, non-critical points with small values of $||\nabla f(\boldsymbol{x})||$ are interpreted as ``near critical''. From Lemma \ref{lem:diverge}, in the multi-objective setting we can analogously interpret non-critical points as ``near critical'' if either for some $i=1,...,m$ the value of $||\nabla f_i(\boldsymbol{x})||$ is small, or, if the value of $\boldsymbol{V}_c(\boldsymbol{x})$ is large. In Figure \ref{fig:EfficientSet}  we illustrate a multi-objective optimization problem and the regions that approximate the efficient set with the measures here delineated.

\begin{figure}
  \includegraphics[width=\linewidth]{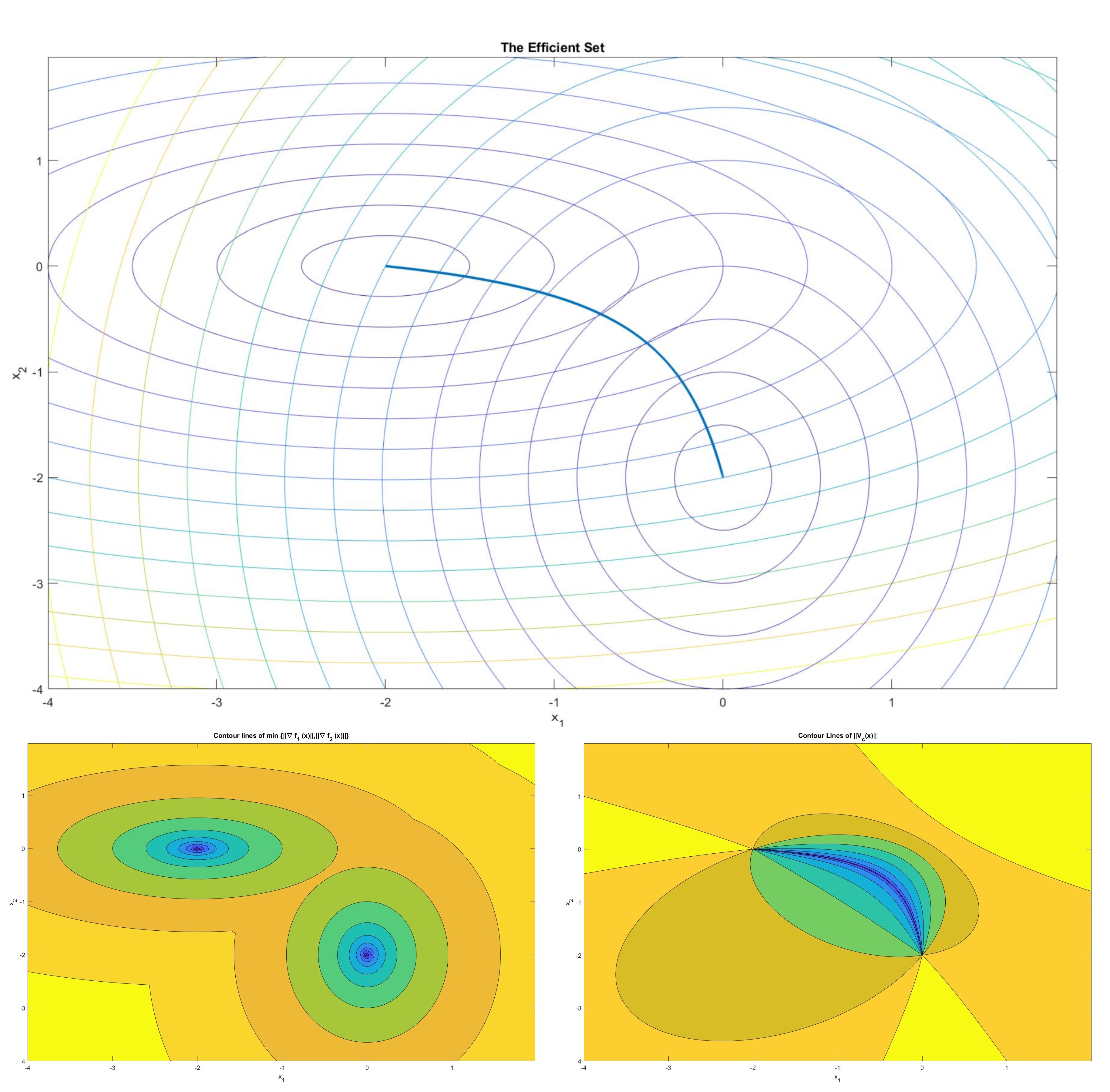}
  \caption{ \label{fig:EfficientSet}The top figure depicts the level curves and the efficient set of the multi-objective optimization problem with objective functions $f_1(\boldsymbol{x}) \equiv (x_1+2)^2 + 3x_2^2$ and $f_2(\boldsymbol{x}) \equiv 3x_1^2 + (x_2+2)^2$. The bottom left depicts the level curves of $\min\{||\nabla f_1(\boldsymbol{x})||;||\nabla f_1(\boldsymbol{x})||\}$ and the bottom right depicts the level curves of $||\boldsymbol{V}_c(\boldsymbol{x})||$. Both measures are complementary to quantify proximity to critical conditions; the  $\min\{||\nabla f_1(\boldsymbol{x})||;||\nabla f_1(\boldsymbol{x})||\}$ measures proximity to local minima of the mono objective sub-problems, and, $||\boldsymbol{V}_c(\boldsymbol{x})||$ measures proximity to intermediate solutions of the multi-objective problem.}
\end{figure}

\paragraph{On the geometric properties and robustness of $\boldsymbol{V}_c$} The central descent direction in addition to providing a measure of proximity to critical conditions that is complimentary to the measure used in mono-objective optimization, it also enjoys of several geometric and robustness properties. Here we highlight a few; the first is that 

\begin{proposition} \label{prop:distance_max}
The descent direction $\boldsymbol{V}_c(\boldsymbol{x})/||\boldsymbol{V}_c(\boldsymbol{x})||$ is the unit vector maximally distant from non-descent directions $\mathbb{R}^n\setminus \{\boldsymbol{v} \text{ st. } \nabla f_i (\boldsymbol{x})^T\boldsymbol{v}\leq 0 \text{ for i = 1,...,m}\}$.
\end{proposition}
\begin{proof}
Given a unit vector $\boldsymbol{u}$ in $\{\boldsymbol{v} \text{ st. } \nabla f_i (\boldsymbol{x})^T\boldsymbol{v}\leq 0 \text{ for i = 1,...,m}\}$, the $L_2$ distance of $\boldsymbol{u}$ to $\mathbb{R}^n\setminus \{\boldsymbol{v} \text{ st. } \nabla f_i (\boldsymbol{x})^T\boldsymbol{v}\leq 0 \text{ for i = 1,...,m}\}$ is given by  $d(\boldsymbol{u}) = \min_{i = 1,...,m} \nabla f_i (\boldsymbol{x})^T\boldsymbol{u} / ||\nabla f_i (\boldsymbol{x})||$. Therefore, the maximization problem that defines the unit vector maximally distant from non-descent directions can be expressed as
$$\begin{array}{ll}
 \max_{z,\boldsymbol{u}} & z \\
 \text{st.} & z ||\nabla f_i (\boldsymbol{x})||\leq \nabla f_i(\boldsymbol{x})^T \boldsymbol{u} \text{ for } i = 1,...,m; \\
 & ||\boldsymbol{u}|| = 1;
 \end{array} $$ 
 which, by defining $\boldsymbol{V} \equiv \boldsymbol{u}/-z $ the objective $\max z$ turns out to be equivalent to $\min ||\boldsymbol{V}||_2$, and, the constraint $z ||\nabla f_i (\boldsymbol{x})||\leq \nabla f_i(\boldsymbol{x})^T \boldsymbol{u} \text{ for } i = 1,...,m;$ turns out to be equivalent to $- ||\nabla f_i (\boldsymbol{x})||\geq \nabla f_i(\boldsymbol{x})^T \boldsymbol{V} \text{ for } i = 1,...,m$. The remaining constraint of $||\boldsymbol{u}|| = 1$  can be dropped by recognizing that $|z| =1/||\boldsymbol{V} ||$ is non-restrictive on the remaining variables. 
\end{proof}

Proposition \ref{prop:distance_max} ensures that the central descent direction is maximally distant from non-descent directions as measured by the $L_2$ norm. As a consequence, numerical errors or approximations in the calculation of $\boldsymbol{V}_c/||\boldsymbol{V}_c||$ are less likely to produce a non-descent direction when compared to any other descent directions in the negative cone of the gradients. A similar claim can be done if only orthogonal perturbations to the descent direction are considered, that is, the direction $\boldsymbol{V}_c/||\boldsymbol{V}_c||$ also happens to be maximally distant from non-descent directions if the distance of $\boldsymbol{u}$ in $\{\boldsymbol{v} \text{ st. } \nabla f_i (\boldsymbol{x})^T \boldsymbol{v}\leq 0 \text{ for i = 1, ..., m} \}$ to non descent directions $\mathbb{R}^n\setminus \{\boldsymbol{v} \text{ st. } \nabla f_i (\boldsymbol{x})^T\boldsymbol{v}\leq 0 \text{ for i = 1,...,m}\}$ is measured only over the orthogonal plane defined by $\{\boldsymbol{x} \text{ st. } (\boldsymbol{x} - \boldsymbol{u})^T\boldsymbol{u} = 0\}$. Hence, the direction $\boldsymbol{V}_c/||\boldsymbol{V}_c||$ seems to be a natural choice if the descent direction $\boldsymbol{d}$ in line 3 of Algorithm \ref{alg:descent_direction} is obtained through approximations rather than exact computations. 

\begin{proposition} \label{prop:monotone_transform}
 Given a collection objective functions $f_1,...,f_m$ and a collection of strictly increasing and differentiable transformations $g_1,...,g_m:\mathbb{R}\to \mathbb{R}$, the central descent direction $\boldsymbol{V}_c(\boldsymbol{x})$ calculated with respect to the objective functions $f_1,...,f_m$ is equal to the central descent direction calculated with respect to the objective functions $g_1\circ f_1,...,g_m \circ f_m$.
\end{proposition}
\begin{proof}
The proof follows immediately from the chain rule: $\tfrac{\partial}{\partial x_j} \left( g_i( f_i (\boldsymbol{x}))\right) = g'_i(f_i(\boldsymbol{x}))\cdot \tfrac{\partial}{\partial x_j} f_i(\boldsymbol{x})$. Applying the chain rule to the central descent direction on the monotonically transformed problem we find that arg$\min \{ \tfrac{1}{2} ||\boldsymbol{V}||^2 \text{ st. } g'_i(f_i(\boldsymbol{x}))\nabla f_i(\boldsymbol{x})^T\boldsymbol{V} \leq -||g'_i(f_i(\boldsymbol{x}))\nabla f_i(\boldsymbol{x})||\}$ is equal to arg$\min \{ \tfrac{1}{2} ||\boldsymbol{V}||^2 \text{ st. } \nabla f_i(\boldsymbol{x})^T\boldsymbol{V} \leq -||\nabla f_i(\boldsymbol{x})||\}$ since the terms $g'_i(f_i(\boldsymbol{x}))$ cancel out.
\end{proof}
Thus, Proposition \ref{prop:monotone_transform} ensures that the central descent direction is invariant to changes on the scales of the objective functions. In contrast, the steepest descent direction in (\ref{eq:steepest_descent}) is not only sensitive to monotone transformations, but also, sensitive to linear transformations of the objectives; i.e. the direction $\boldsymbol{V}_s(\boldsymbol{x})$ obtained by considering the objective functions in $F(\boldsymbol{x}) \equiv [f_1(\boldsymbol{x}), f_2(\boldsymbol{x})]^T$ is not the same as the one obtained by considering $F(\boldsymbol{x}) \equiv [f_1(\boldsymbol{x}), \kappa f_2(\boldsymbol{x})]^T$ for $\kappa>0$ with $\kappa \neq 1$. This property can dramatically warp the path taken by the descent direction algorithm if the scales are not a-priori fine tuned, a requirement that might be difficult to meet since multi-objective optimization is adopted precisely when the relative weights of the objectives is unknown. We illustrate this contrast between $\boldsymbol{V}_s$ and $\boldsymbol{V}_c$ in Figures \ref{fig:ContourCauchy} to \ref{fig:Vs_and_Vc}. Figure \ref{fig:ContourCauchy} depicts the level sets of the measure of proximity to critical conditions induced by $\boldsymbol{V}_s$ as well as the stream-lines (the curves produced by ``releasing a particle to flow in the direction of $\boldsymbol{V}_s$'' ) of the problem illustrated in Figure \ref{fig:EfficientSet} when the objective functions are multiplied by different constants. The warping effect can over-weigh one objective over the other producing contour-lines and stream-lines that can even parallel the efficient set making the descent direction algorithm unnecessarily ``go arround'' close-by solutions to converge at distant ones. And, this effect can be dramatically intensified with different monotone transformations. Figure \ref{fig:StreamLinesVc} shows both the contour-lines as well as the stream-lines induced by $\boldsymbol{V}_c$, and, as can be seen the curves follow a more natural path towards a close-by efficient solution irrespective of the scale adopted in the representation of the objective functions. Finally, Figure \ref{fig:Vs_and_Vc} depicts $\boldsymbol{V}_s$ and $\boldsymbol{V}_c$ for varying values of $||\nabla f_i(\boldsymbol{x})||$ for $i=1,2$ as well as an approximate the construction of $\boldsymbol{V}_c$.

\begin{figure}
  \includegraphics[width=\linewidth]{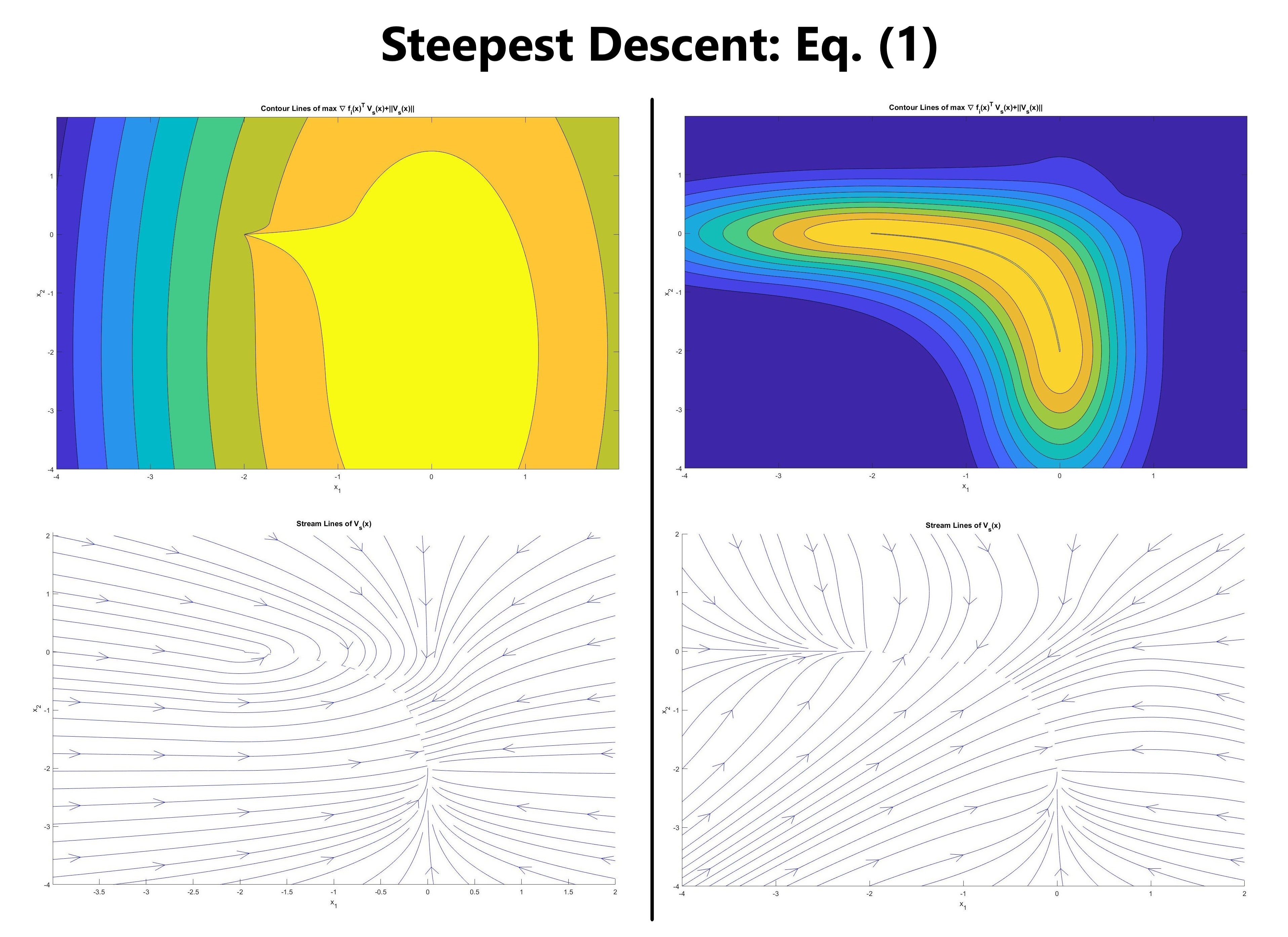}
  \caption{\label{fig:ContourCauchy}The top two images depict the level sets of the measure of proximity induced by the steepest descent (\ref{eq:steepest_descent}), i.e the level sets of $\min_{\boldsymbol{V}\in\mathbb{R}^n} \max_{i = 1,...,m}\nabla f_i(\boldsymbol{x})^T \boldsymbol{V}  +\frac{1}{2}||\boldsymbol{V}||^2$. Since the steepest descent direction is not metric-independent, the level sets of the measure of proximity can be warped (as it is on the left side) when the objectives are not measured with some comparable scale (as on the right side). The path taken by the steepest descent method is affected by the warping and can produce curves that are arbitrarily closet to the efficient set and yet parallel the efficient set. The path of the steepest descent method (when the steps are sufficiently small) will follow the stream lines depicted on the two lower images.   }
\end{figure}

\begin{figure}
  \includegraphics[width=.9\linewidth]{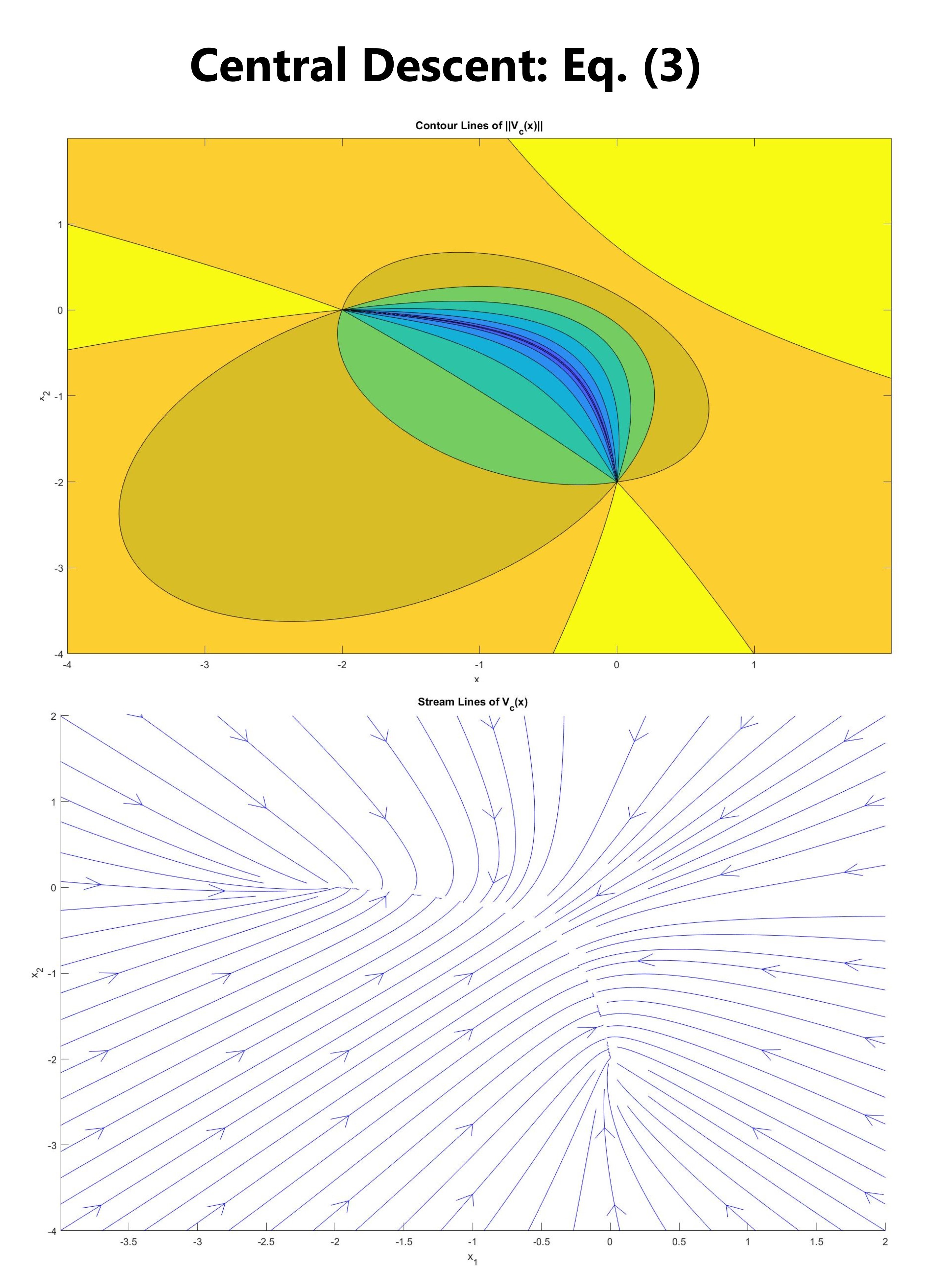}
  \caption{\label{fig:StreamLinesVc}The top image depicts the level sets of $||\boldsymbol{V}_c(\boldsymbol{x})||$ (the same as the bottom right of Figure \ref{fig:EfficientSet}); and, the bottom image depicts the stream lines induced by the central descent direction. Notice that since the central descent direction is unaffected by the rescaling of the objective functions, these level sets are not warped by different scales/representations of the same collection of objectives. Also, notice that proximity to the minima of the mono-objective functions is not measured by this metric, as the level curves meet at a sharp angle at each individual minima; instead it measures the angle between the gradients independently of their norms. This produces (on the bottom figure) stream lines that are metric independent and are not warped by arbitrary scaling choices.}
\end{figure}

\begin{figure}
  \includegraphics[width=\linewidth]{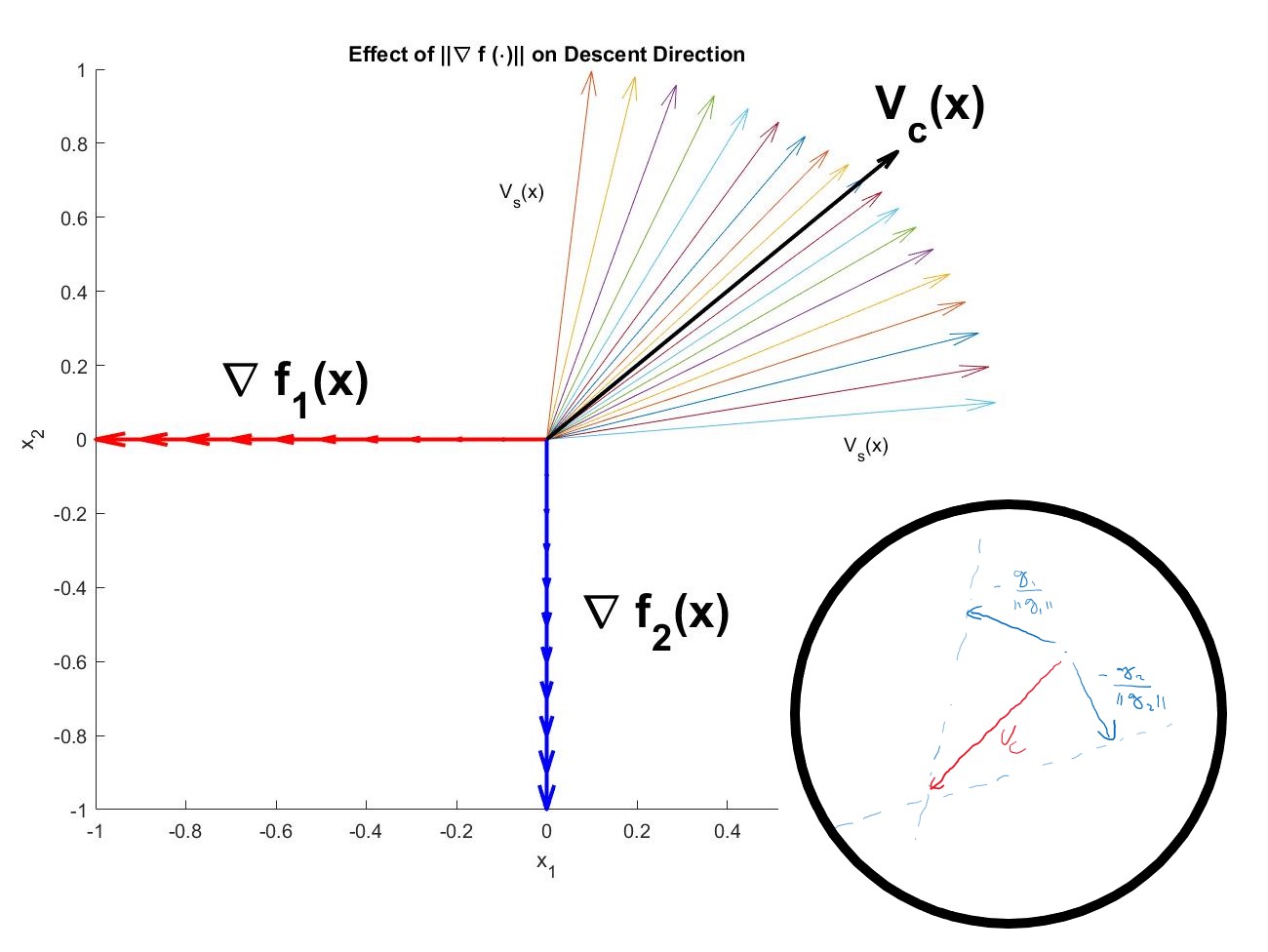}
  \caption{\label{fig:Vs_and_Vc} The background image depicts the effect of changing the scales of the objective functions. The adoption of different scales on the objective functions alter the sizes of the gradients $\boldsymbol{g}_i \equiv \nabla f_i (\boldsymbol{x})$ for $i=1,2$, but not the directions of the gradients. Both size and direction of the central descent are unaffected by these changes in scale, however the direction of $\boldsymbol{V}_s$  described in (\ref{eq:steepest_descent}) is affected by these changes in scale. The size of $\boldsymbol{V}_s$ is also affected by such changes, however, here we only depict the projections of the vectors $\boldsymbol{V}_s$ and $\boldsymbol{V}_c$ to the unit ball for ease of visualization. The smaller image in the bottom right circle depicts the geometric construction of the vector $\boldsymbol{V}_c$.  The central descent direction can be obtained by bisecting the angle between $-\nabla f_1(\boldsymbol{x})$ and $-\nabla f_2(\boldsymbol{x})$; and, the value of $||\boldsymbol{V}_c(\boldsymbol{x})||$ is obtained by intersecting perpendicular lines from $\boldsymbol{g}_1/||\boldsymbol{g}_1||$ and $\boldsymbol{g}_2/||\boldsymbol{g}_2||$. The longer the vector depicting $\boldsymbol{V}_c(\cdot)$ the closer $\boldsymbol{g}_1$ and $\boldsymbol{g}_2$ are from pointing in opposite directions.  }
\end{figure}

\section{Incremental central descent method} \label{sec:incremental_descent}

In this section we delineate our main results. Here we describe a multi-objective incremental descent approach and analyse both global guarantees as well as theoretical speed of convergence. The choice of performing incremental approximations to the central descent, rather than the steepest descent, seems to be well justified since, as demonstrated in the last section,  the central descent direction is robust under monotone transformations of the objective functions (i.e. no a priori fine-tuning of the scales is necessary to avoid warping of the directions) and more importantly, it is the direction farthest away from the non-descent directions. Hence, small errors in the estimate $\hat{\boldsymbol{d}}$ of $\boldsymbol{V}_c/||\boldsymbol{V}_c||$ are less likely to produce a non-descent direction when compared to any other unit vector in $\{\boldsymbol{v} \text{ st. } \nabla f_i (\boldsymbol{x})^T\boldsymbol{v} \leq 0 \text{ for i = 1,...,m}\}$. 

First we provide our global guarantees for vanishing step-sizes:

\subsection{Global guarantees for vanishing step-sizes} 
In the following, all that is assumed is that the sequence of step-sizes $\{\alpha_k\}_{k = 1,2,...}$ satisfy $\alpha_k>0$ for all $k$ and $\lim_{k\to \infty}\alpha_k = 0$ and $\sum_{k \in \mathbb{N}}\alpha_k = \infty$. This condition can be ensured by pre-specifying the step-sizes (e.g. $\alpha_k = 1/k$) or by performing a search in vanishing intervals. The algorithm described bellow is initiated with arbitrary values for $\boldsymbol{\hat{x}}\in \mathbb{R}^n$ and non-null $\boldsymbol{\hat{g}}_1,...,\boldsymbol{\hat{g}}_m \in \mathbb{R}^n$.

\begin{algorithm}[H]
\SetAlgoLined
initialize $\boldsymbol{\hat{x}}\in \mathbb{R}^n$ and $\boldsymbol{\hat{g}}_1,...,\boldsymbol{\hat{g}}_m \in \mathbb{R}^n\setminus \boldsymbol{0}$\ \ and \ \  $k\leftarrow 1$ \ \ and \ \ $\alpha \leftarrow \alpha_k$ \ \ and \ \ $t \leftarrow 1$\;
 \While{True}{
  $\boldsymbol{\hat{g}}_t\leftarrow \nabla f_t(\hat{x})$ \ \ and if $\boldsymbol{\hat{g}}_t$ is null then stop\;
    $\boldsymbol{\hat{V}} \leftarrow \text{argmin}\{||\boldsymbol{V}||^2 \text{ st. } \boldsymbol{\hat{g}}_i^T\boldsymbol{V} \leq -||\boldsymbol{\hat{g}}_i|| \text{ for } i = 1,...,m \}$ and if  $\boldsymbol{\hat{V}}$ is empty then stop\;
    $\hat{\boldsymbol{x}} \leftarrow \hat{\boldsymbol{x}} + \alpha \boldsymbol{\hat{V}}/||\boldsymbol{\hat{V}}||$ \ \ and \ \ $k\leftarrow k + 1$ \ \ and \ \ $\alpha \leftarrow \alpha_k$ \ \ and \ \ $t \leftarrow (t+1) \text{ mod } m$\;
 }
 \caption{\label{alg:central_basic}Increment central descent}
\end{algorithm}

We will now prove the convergence properties of Algorithm \ref{alg:central_basic}.

\begin{theorem}\label{the:incremental_central_descent}
Algorithm \ref{alg:central_basic} produces a subsequence of estimates $\{\boldsymbol{\hat{x}}_{k_j}\}_{j = 1,2,...}$ such that as $j\to \infty$ either (i) the gradient $\nabla f_i(\boldsymbol{\hat{\boldsymbol{x}}}_{k_j})$ vanishes for some $i = 1,...,m$; or (ii) the the central descent direction $\boldsymbol{V}_c(\boldsymbol{\hat{\boldsymbol{x}}}_{k_j})$ is unbounded; or (iii) all functions are unbounded from bellow and decrease indefinitely:
\begin{equation}\label{eq:central_guarantees}
\begin{array}{llllllll}
(i) & \lim_{j\to \infty}||\nabla f_i(\boldsymbol{\hat{\boldsymbol{x}}}_{k_j})|| = 0 & \text{ or } & (ii) & ||\boldsymbol{V}_c(\boldsymbol{\hat{\boldsymbol{x}}}_{k_j})|| \to \infty & \text{ or } & (iii) & \lim_{j\to \infty}f_i(\hat{\boldsymbol{x}}_{k_j}) = -\infty.\\
& \text{for some i = 1,...,m} & & & & & & \text{for all i = 1,...,m}
\end{array}
\end{equation}   
\end{theorem}
\begin{proof}
In this proof we will show that if (i) does not occur, then either (ii) or (iii) must occur. The negation of (i)  implies that $||\nabla f_i(\boldsymbol{\hat{x}}_{k})||$ is lower-bounded by some positive real value $c_1$ for all $i=1,...,m$ and all $k \in \mathbb{N}$ is thus assumed henceforth.

Notice that for any iteration $k\geq m+1$ and any $i = 1,...,m$ we have that $||\nabla f_i(\boldsymbol{\hat{x}}_k) - \boldsymbol{\hat{g}}_{i,k}||= ||\nabla f_i(\boldsymbol{\hat{x}}_k) - \nabla f_i(\boldsymbol{\hat{x}}_\tau)||\leq L ||\boldsymbol{\hat{x}}_k-\boldsymbol{\hat{x}}_\tau||$ for some $\tau$ between $k-m$ and $k$; and furthermore $ ||\boldsymbol{\hat{x}}_k-\boldsymbol{\hat{x}}_\tau|| \leq  ||\boldsymbol{\hat{x}}_k-\boldsymbol{\hat{x}}_{k-1}||+ ||\boldsymbol{\hat{x}}_{k-1}-\boldsymbol{\hat{x}}_{k-2}||+...+ ||\boldsymbol{\hat{x}}_{k-m+1}-\boldsymbol{\hat{x}}_{k-m}||=\sum_{k-m}^{k}\alpha_j$; and therefore
\begin{equation} \label{eq:grad_converge}
||\nabla f_i(\boldsymbol{\hat{x}}_k) - \boldsymbol{\hat{g}}_{i,k}||\leq L\sum_{k-m}^{k}\alpha_j \ \ \text{ for any }k\geq m+1 \text{ and }i = 1,...m;
\end{equation}
and thus $||\nabla f_i(\boldsymbol{\hat{x}}_k) - \boldsymbol{\hat{g}}_{i,k}||$ goes to zero as $k$ increases. 
With this fact established we will now analyse two complementary cases:\\
\textbf{Case 1.}  There exists a subsequence $\{k_j\}_{j=1,...,\infty}$  for which $\lim_{j\to\infty}||\boldsymbol{\hat{V}}_{k_j}|| = \infty$.\\
\textbf{Case 2.}  There exists an upper-bound $c_2>0$ for which  $||\boldsymbol{\hat{V}}_k||\leq c_2$ for every $k\in \mathbb{N}$.\\

\textbf{Analysis of Case 1.}

Under the conditions of Case 1  all that is needed is to show that when $||\boldsymbol{\hat{V}}_{k_j}||\to \infty$ for $j\to \infty$, then the points $\boldsymbol{\hat{x}}_k$ produce a subsequence of central descent directions $\boldsymbol{V}_k\equiv \text{argmin}\{||\boldsymbol{V}||^2 \text{ st. } f_i(\boldsymbol{\hat{x}_k})^{T}V\leq -||\nabla f_i(\boldsymbol{\hat{x}}_k)|| \}$ that diverge. For this, notice that over the subsequence in which $||\boldsymbol{\hat{V}}_{k_j}||$ diverges we have: (A) the vectors $\hat{\boldsymbol{g}}_{i,k}/||\hat{\boldsymbol{g}}_{i,k}||$ are contained in a unit ball, and thus, there must exist a converging subsequence where $\lim_{k\to\infty}\hat{\boldsymbol{g}}_{i,k}/||\hat{\boldsymbol{g}}_{i,k}|| = \boldsymbol{u}_i$; and (B) $||\nabla f_i(\boldsymbol{\hat{x}}_{k})||$ is lower-bounded by some positive real value $c_1$ for all $i=1,...,m$ and all $k \in \mathbb{N}$; and (C) the value of $||\nabla f_i(\boldsymbol{\hat{x}}_k) - \boldsymbol{\hat{g}}_{i,k}||$ goes to zero as $k$ increases. The combination of (A), (B) and (C) imply that $\nabla f_i(\boldsymbol{\hat{x}}_k)/||\nabla f_i(\boldsymbol{\hat{x}}_k)||$ also converges to $\boldsymbol{u}_i$, and thus, by Lemma \ref{lem:lower} part 1 we conclude that the points $\boldsymbol{\hat{x}}_k$ produce a subsequence of central descent directions $\boldsymbol{V}_k\equiv \text{argmin}\{||\boldsymbol{V}||^2 \text{ st. } f_i(\boldsymbol{\hat{x}}_k)^{T}V\leq -||\nabla f_i(\boldsymbol{\hat{x}}_k)|| \}$ that diverge.\\

\textbf{Analysis of Case 2.}

Lipschitz continuity of the gradients ensures that for each i = 1,...,m we have 
\begin{equation} f_i(\boldsymbol{\hat{x}}_{k+1})-f_i(\boldsymbol{\hat{x}}_{k}) \leq \nabla f_i(\boldsymbol{\hat{x}}_{k})^{T}(\boldsymbol{\hat{x}}_{k+1}-\boldsymbol{\hat{x}}_{k}) +\tfrac{1}{2} L||\boldsymbol{\hat{x}}_{k+1}-\boldsymbol{\hat{x}}_{k}||^2;
\end{equation}
and thus
$$f_i(\boldsymbol{\hat{x}}_{k+1})-f_i(\boldsymbol{\hat{x}}_{k}) \leq \alpha_k \nabla f_i(\boldsymbol{\hat{x}}_{k})^{T} \boldsymbol{\hat{V}}_k/||\boldsymbol{\hat{V}}_k|| +\tfrac{1}{2} L\alpha_k^2 = \alpha_k [\boldsymbol{\hat{g}}_{i,k}+(\nabla f_i(\boldsymbol{\hat{x}}_{k})-\boldsymbol{\hat{g}}_{i,k})]^{T} \boldsymbol{\hat{V}}_k/||\boldsymbol{\hat{V}}_k|| +\tfrac{1}{2} L\alpha_k^2$$
$$\leq  \alpha_k \boldsymbol{\hat{g}}_{i,k}^{T} \boldsymbol{\hat{V}}_k/||\boldsymbol{\hat{V}}_k||+\alpha_k L\sum_{k-m}^{k}\alpha_j +\tfrac{1}{2} L\alpha_k^2\leq -\alpha_k ||\boldsymbol{\hat{g}}_{i,k}||/||\boldsymbol{\hat{V}}_k||+\alpha_k L\sum_{k-m}^{k}\alpha_j +\tfrac{1}{2} L\alpha_k^2;$$
where the first inequality of the second line is a consequence of (\ref{eq:grad_converge}) and the second inequality is a consequence of the definition of $\boldsymbol{\hat{V}}_k$. Now, using (\ref{eq:grad_converge}) a second time on the first term we obtain $-||\boldsymbol{\hat{g}}_{i,k}||\leq - ||\nabla f_i(\boldsymbol{\hat{x}}_k)||+L\sum_{k-m}^{k}\alpha_j$ and thus $-\alpha_k ||\boldsymbol{\hat{g}}_{i,k}||/||\boldsymbol{\hat{V}}_k||\leq - \alpha_k||\nabla f_i(\boldsymbol{\hat{x}}_k)||/||\boldsymbol{\hat{V}}_k||+\alpha_k L(\sum_{k-m}^{k}\alpha_j)/||\boldsymbol{\hat{V}}_k||$, and since by construction  $||\boldsymbol{\hat{V}}_k||\geq 1$ (when it exists) and by assumption $||\nabla f_i(\boldsymbol{\hat{x}}_k)||\geq c_1$, then we conclude that the first term is upper-bounded by $- \alpha_k c_1/||\boldsymbol{\hat{V}}_k||+\alpha_k L\sum_{k-m}^{k}\alpha_j$. Hence for every $i= 1,...,m$ and every $k\geq m+1$ we have
\begin{equation} \label{eq:clean_upper_bound}
f_i(\boldsymbol{\hat{x}}_{k+1})-f_i(\boldsymbol{\hat{x}}_{k}) \leq \alpha_k\left(  2L\sum_{k-m}^{k}\alpha_j +\tfrac{1}{2} L\alpha_k- \frac{c_1}{||\boldsymbol{\hat{V}}_k||}\right).
\end{equation}

Furthermore, under the conditions of Case 2, there exists a constant $c_2>0$ such that $||\boldsymbol{\hat{V}}_k||\leq c_2$ and thus $-c_1/||\boldsymbol{\hat{V}}_k||\leq -c_1/c_2$. Inserting this back into (\ref{eq:clean_upper_bound}) we obtain
$$ f_i(\boldsymbol{\hat{x}}_{k+1})-f_i(\boldsymbol{\hat{x}}_{k}) \leq \alpha_k\left(  2L\sum_{k-m}^{k}\alpha_j +\tfrac{1}{2} L\alpha_k- c_1/c_2\right).$$
Now notice that the first term within the brackets vanishes with increasing values of $k$, and thus for sufficiently large $k$ we have $2L\sum_{k-m}^{k}\alpha_j +\tfrac{1}{2} L\alpha_k\leq \tfrac {1}{2} c_1 /c_2$. Hence, for all $i=1,...,m$ and for all $k\geq \bar{k}$, for some $\bar{k}\in\mathbb{N}$, we have
\begin{equation}\label{eq:minimum_decrease}
 f_i(\boldsymbol{\hat{x}}_{k+1})-f_i(\boldsymbol{\hat{x}}_{k}) \leq -\tfrac{1}{2}\alpha_kc_1/c_2.
\end{equation}
Summing up the terms for $k\geq \bar{k}$ in equation (\ref{eq:minimum_decrease}) we obtain:
$$
 \left[\lim_{k\to \infty}f_i(\boldsymbol{\hat{x}}_{k})\right]-f_i(\boldsymbol{\hat{x}}_{\bar{k}}) \leq -\tfrac{1}{2}\frac{c_1}{c_2}\sum_{k\geq\bar{k}}\alpha_k = -\infty.$$
In this case all functions are unbounded and the sequence produces a subsequence of points in which all functions are simultaneously decreased indefinitely. This concludes our proof. 
\end{proof}

\subsection{A global $1/\sqrt{k}$ convergence of critical conditions for bounded functions} \label{sec:lowest_descent}
Assuming all $f_i$ are bounded from bellow we include one additional gradient computation per iteration and an Armijo-type sufficient decrease condition with some pre-specified parameter $\beta\in (0,1)$ for the construction of the step size:

\begin{algorithm}[H]
\SetAlgoLined
initialize $\boldsymbol{\hat{x}}\in \mathbb{R}^n$ and $\boldsymbol{\hat{g}}_1,...,\boldsymbol{\hat{g}}_m \in \mathbb{R}^n\setminus \boldsymbol{0}$\ \ and \ \  $k\leftarrow 1$ \ \ and \ \ $j \leftarrow 1$ \ \ and \ \ $t \leftarrow 2$\;
 \While{While True}{
  $\boldsymbol{\hat{g}}_{j}\leftarrow \nabla f_j(\hat{x})$; $\boldsymbol{\hat{g}}_t\leftarrow \nabla f_t(\hat{x})$ \ \ and if $\boldsymbol{\hat{g}}_j$ or $\boldsymbol{\hat{g}}_t$ is null then stop\;
    $\boldsymbol{\hat{V}} \leftarrow \text{argmin}\{||\boldsymbol{V}||^2 \text{ st. } \boldsymbol{\hat{g}}_i^T\boldsymbol{V} \leq -||\boldsymbol{\hat{g}}_i|| \text{ for } i = 1,...,m \}$ and if  $\boldsymbol{\hat{V}}$ is empty then stop\;
    $\alpha \leftarrow \text{max}_{l = 0,1,...,\infty}\alpha = (1/2)^l$ st. $f_j(\hat{\boldsymbol{x}} + \alpha \boldsymbol{\hat{V}}/||\boldsymbol{\hat{V}}||) - f_j(\hat{\boldsymbol{x}})\leq \beta \alpha \boldsymbol{\hat{g}_j}^T \boldsymbol{\hat{V}}/||\boldsymbol{\hat{V}}||$\;
    $\hat{\boldsymbol{x}} \leftarrow \hat{\boldsymbol{x}} + \alpha \boldsymbol{\hat{V}}/||\boldsymbol{\hat{V}}||$ \ \ and \ \ $k\leftarrow k + 1$ \; 
    choose $t\neq j$ between $1$ and $m$  \ \ and \ \ if $f_t(\hat{\boldsymbol{x}})< f_j(\hat{\boldsymbol{x}})$ then swap $j$ and $t$\;
 }
 \caption{\label{alg:lowest_central_descent}Increment central descent w/ inexact line-search}
\end{algorithm}

\begin{lemma}
\begin{equation}
\alpha\geq \alpha_{\min} \equiv \min\left\{\tfrac{1-\beta}{2L_j},1\right\}||\boldsymbol{\hat{g}}_j||/||\boldsymbol{\hat{V}}||
\end{equation}
\end{lemma}
\begin{proof}
When $2\alpha$ does not satisfy sufficient decrease condition we have 
$$f_j(\hat{\boldsymbol{x}}+2\alpha \boldsymbol{\hat{V}}/||\boldsymbol{\hat{V}}|| ) - f_j(\hat{\boldsymbol{x}}) > \beta 2\alpha\nabla f_j(\hat{\boldsymbol{x}})^T \boldsymbol{\hat{V}}/||\boldsymbol{\hat{V}}||$$
From Lipschits condition:
$$f_j(\hat{\boldsymbol{x}}+2\alpha \boldsymbol{\hat{V}} /||\boldsymbol{\hat{V}}|| ) - f_j(\hat{\boldsymbol{x}})\leq 2 \alpha \nabla f_j(\hat{\boldsymbol{x}})^T\boldsymbol{\hat{V}}/||\boldsymbol{\hat{V}}|| + \tfrac{L_j}{2}||2\alpha \frac{\boldsymbol{\hat{V}}}{||\boldsymbol{\hat{V}}||}||^2$$
$$\implies 2\alpha(1-\beta)\nabla f_j(\hat{\boldsymbol{x}})^T\boldsymbol{\hat{V}}/||\boldsymbol{\hat{V}}|| + 2L_j\alpha^2\geq 0$$
$$\implies -L_j\alpha \leq (1-\beta)\nabla f_j(\hat{\boldsymbol{x}})^T\boldsymbol{\hat{V}}/||\boldsymbol{\hat{V}}||\leq  - (1-\beta) ||\nabla f_j(\hat{\boldsymbol{x}})||/||\boldsymbol{\hat{V}}||$$
$$\implies \alpha \geq \frac{1-\beta}{L_j}||\nabla f_j(\hat{\boldsymbol{x}})||/||\boldsymbol{\hat{V}}|| = \tfrac{1-\beta}{2L_j}||\boldsymbol{\hat{g}}_j||/||\boldsymbol{\hat{V}}||$$
\end{proof}
\begin{theorem} \label{the:lowest_incremental_central_descent}
Suppose all functions $f_i$ are bounded from bellow and let $f^{\min}$ be a lower bound on all $f_i$. The incremental central descent method with inexact line searching generates a sequence such that:
\begin{equation} \label{eq:lowest_cost}
\min_{0\leq l\leq k-1} \frac{\min_{t = 1,...,m}||\nabla f_t(\hat{\boldsymbol{x}}_l)||}{||\hat{\boldsymbol{V}}_l||} \leq \sqrt{\frac{1}{k}\left(\frac{f_1(\hat{\boldsymbol{x}}_0)- f^{\min}}{\min\left\{\frac{\beta(1-\beta)}{2L},\beta\right\}}\right)}.
\end{equation}
\end{theorem}
\begin{proof}
In each iteration the following inequality holds for the objective function indexed by $j$
$$f_j(\hat{\boldsymbol{x}}^{k+1}) - f_j(\hat{\boldsymbol{x}}^k)\leq \beta \alpha^k \nabla f_j(\hat{\boldsymbol{x}}^k)^T\boldsymbol{\hat{V}}/||\boldsymbol{\hat{V}}||\leq -\beta \alpha^k ||\nabla f_j(\hat{\boldsymbol{x}}^k)||/||\boldsymbol{\hat{V}}^k||$$
and therefore
$$f_j(\hat{\boldsymbol{x}}^k)- f_j(\hat{\boldsymbol{x}}^{k+1}) \geq \beta \alpha^k ||\nabla f_j(\hat{\boldsymbol{x}}^k)||/||\boldsymbol{\hat{V}}|| \geq \beta \alpha_{\min} ||\nabla f_j(\hat{\boldsymbol{x}}^k)||/||\boldsymbol{\hat{V}}^k|| $$
$$ =  \beta  \min\left\{\frac{1-\beta}{2L},1\right\}||\nabla f_j(\hat{\boldsymbol{x}}^k)||^2/||\boldsymbol{\hat{V}}||^2 \geq   \min\left\{\frac{\beta(1-\beta)}{2L},\beta\right\}\min_{t=1,...,m}||\nabla f_t(\hat{\boldsymbol{x}}^k)||^2/||\boldsymbol{\hat{V}}||^2. $$
Now, let the index $j$ of iteration $k'$ be represented by $j(k')$. Notice that because of line 7 of Algorithm \ref{alg:lowest_central_descent} we have $f_{j(k')}(\hat{\boldsymbol{x}}^{k'+1})\geq f_{j(k'+1)}(\hat{\boldsymbol{x}}^{k'+1}) $. Thus,

$$f_{j(k')}(\hat{\boldsymbol{x}}^{k'})- f_{j(k'+1)}(\hat{\boldsymbol{x}}^{k'+1}) \geq f_{j(k')}(\hat{\boldsymbol{x}}^{k'})- f_{j(k')}(\hat{\boldsymbol{x}}^{k'+1}) $$
and therefore, between iterations, we obtain a decrease in the function values indexed by $j$ lower-bounded by
\begin{equation} \label{eq:lowest_decrease}
f_{j(k')}(\hat{\boldsymbol{x}}^{k'})- f_{j(k'+1)}(\hat{\boldsymbol{x}}^{k'+1}) \geq  \min\left\{\frac{\beta(1-\beta)}{2L},\beta\right\}\min_{t=1,...,m}||\nabla f_t(\hat{\boldsymbol{x}}^{k'})||^2/||\boldsymbol{\hat{V}}^{k'}||^2. \end{equation}

By summing the terms in equation (\ref{eq:lowest_decrease}) for varying for values of $k'$ between $0$ and $k-1$ we find:
$$f_{j(0)}(\hat{\boldsymbol{x}}^{0})- f_{j(k-1)}(\hat{\boldsymbol{x}}^{k-1}) \geq  \min\left\{\frac{\beta(1-\beta)}{2L},\beta\right\}\sum_{l = 0}^{k-1}\min_{t=1,...,m}||\nabla f_t(\hat{\boldsymbol{x}}^l)||^2/||\boldsymbol{\hat{V}}^l||^2.$$
And therefore
$$f_{j(0)}(\hat{\boldsymbol{x}}^{0})- f^{\min} \geq  k \min\left\{\frac{\beta(1-\beta)}{2L},\beta\right\}\min_{0\leq l \leq k-1}\min_{t=1,...,m}||\nabla f_t(\hat{\boldsymbol{x}}^l)||^2/||\boldsymbol{\hat{V}}^l||^2$$

$$\implies \frac{1}{k}\left( \frac{f_{1}(\hat{\boldsymbol{x}}^{0})- f^{\min}}{ \min\left\{\frac{\beta(1-\beta)}{2L},\beta\right\}}\right) \geq  \min_{0\leq l \leq k-1}\min_{t=1,...,m}||\nabla f_t(\hat{\boldsymbol{x}}^l)||^2/||\boldsymbol{\hat{V}}^l||^2.$$
Which completes our proof.  \end{proof}

Theorem \ref{the:lowest_incremental_central_descent} provides, for the first time, a positive answer to question \textbf{Q2}; i.e. it is possible to produce the same $O(1/\sqrt{k})$ query complexity as mono-objective optimization with a global cost that is unaffected by increasing values of $m$. The method delineated makes use of two gradients per iteration and one mono-objective inexact line search irrespective of the value of $m$. Furthermore, since the inexact search problem solved by Algorithm \ref{alg:lowest_central_descent} is identical to the one tackled in \cite{oliveira3}, if line 5 were substituted with modern line search alternatives with sub-logarithmic complexity, such as the ones described in \cite{oliveira1,oliveira3}, we can furthermore guarantee that the inexact Armijo-type search has a vanishing contribution to the overall cost when compared to the gradient computations. 


\section{Discussions} \label{sec:discussions}
Current state-of-the-art solvers for multi-objective optimization problems require computing the gradients of all $m$ objective functions per iteration and one $m$ dimensional line-search to produce a worst case convergence to critical conditions at the rate of $O(1/\sqrt{k})$, where $k$ is the iteration count.  Here, we propose an incremental descent method that achieves the same rate of $O(1/\sqrt{k})$ with at most two gradient computations per iteration and one mono-objective line search; i.e. a reduction in the computational cost by a factor of $m$. Furthermore, unlike other incremental strategies developed in the mono-objective literature which require convexity type requirements, these results are obtained solely under the assumption that the objective functions have  $L-$Lipschitz continuous gradients.

The methods here developed make use of a brand new descent direction much similar to the Cauchy's steepest descent, which we term the \emph{central descent}  direction. The central descent is shown to have improved robustness and geometric guarantees which are not shared by other directions considered so far in the literature. And, when approximated incrementally, it allows the construction of the method here proposed with a computational cost that is unaffected by increasing values of $m$.

\paragraph{Future work} The results here attained produce a convergence with a query complexity  matching that of mono-objective optimization. However, as mentioned in the discussion of \textbf{P1} to \textbf{P3}, since multi-objective optimization can be seen as a relaxation to mono-objective optimization, it is natural to expect that increasing values of $m$ should reduce the overall computational cost of the search, specifically when the dependence on the stopping criteria is made explicit. This is so because a larger area of the solution space is considered near-critical with the increase in the number of objectives. Since in this paper we focused solely on the iteration cost, the dependence on the stopping criteria remained open, and thus, with future work the methods here developed may prove to have a diminishing cost with increasing values of $m$ rather than a fixed cost as the results reported here. Furthermore, our formulation despite having attained a reduced query complexity, it still requires a memory cost that is dependent and increasing with $m$. It might be possible to exploit the fact that each gradient only shows up as a restriction in the formulation of (\ref{eq:central_descent}), and, formulate an update scheme in the computation of $\hat{\boldsymbol{V}}$ that might mitigate and even eliminate the increasing memory  cost with increasing values of $m$.  Another compelling direction of research is to investigate if it is possible to exploit the ``opposite direction'' of scalarization; i.e. to formulate multi-objective relaxations of mono-objective problems with the intent of reducing the overall query complexity of the problem. We are unaware of any research done in this direction and we believe a mapping of the trade-offs associated with such a relaxation might prove to open brand new methods of solving classical problems.

\begin{acks}
This paper was written when the first author was a graduate student at the Federal University of Minas Gerais.
\end{acks}

%
\bibliographystyle{ACM-Reference-Format}
\bibliography{samplebase}

%
\appendix

\section{Proof of Auxiliary Lemma \ref{lem:lower}} \label{sec:aux_lem1}
\begin{proof}
\textbf{Part 1.}
Define $\epsilon(R)$ as any positive value strictly less than $
\epsilon^*(R) \equiv \frac{z(R)+1}{R+1}\min_j\{||\boldsymbol{g}_j|| \text{ for j = 1, ..., m}\}$ where $z(R)$ is uniquely defined as:
\begin{equation} \label{eq:gap}
z(R) \equiv \left\{\begin{array}{ll}
\min_{z\in \mathbb{R}, \boldsymbol{v}\in \mathbb{R}^n} & z \\
\text{st.} & \boldsymbol{g}_i^T\boldsymbol{v} \leq z ||\boldsymbol{g}_i||\text{ for all i = 1,...,m};\\
 & ||\boldsymbol{v}||\leq R.
\end{array}\right.
\end{equation}

Notice that since the intersection  of $\{ \boldsymbol{v}\text{ st. } ||\boldsymbol{v}||\leq R\}$ with $\mathcal{S}$ is empty, it must be that $z(R)$ is strictly greater than $-1$ which implies that $\epsilon^*(R)$ is strictly greater than zero. Furthermore, notice that for any $\boldsymbol{v}$  which satisfies $||\boldsymbol{v}||\leq R$ we will have that for some $i$ between $1$ and $m$ the relation $\boldsymbol{g}_i^T\boldsymbol{v}\geq z(R) ||\boldsymbol{g}_i||$ holds. This is because only for the minimizers $\boldsymbol{v}^*,z^*$ of (\ref{eq:gap}) that $\boldsymbol{g}_i^T\boldsymbol{v}$ equates to $ z(R) ||\boldsymbol{g}_i||$ on (at least) one value of i between 1 and m. For non-optimal values of $\boldsymbol{v}$, there will exist an $i$ where the condition $\boldsymbol{g}_i^T\boldsymbol{v}\leq z(R)||\boldsymbol{g}_i||$ must be broken. Thus, consider a collection of $\hat{\boldsymbol{g}}_i$'s in which $||\hat{\boldsymbol{g}}_i-\boldsymbol{g}_i||\leq \epsilon(R)$ for all $i = 1,...,m$ and any $\boldsymbol{v}$  which satisfies $||\boldsymbol{v}||\leq R$; then,  we have that for some $i$ the following must hold:
\begin{equation} \label{eq:bounding} \hat{\boldsymbol{g}}_i^T \boldsymbol{v} = \boldsymbol{g}_i^T\boldsymbol{v}+ (\hat{\boldsymbol{g}}_i - \boldsymbol{g}_i)^T\boldsymbol{v} \geq z(R) ||\boldsymbol{g}_i||-\epsilon(R)\cdot R.\end{equation}
What we must show is that the right hands side of (\ref{eq:bounding}) is strictly greater than $-||\hat{\boldsymbol{g}}_i||$. This follows from
$$z(R) ||\boldsymbol{g}_i||-\epsilon(R)\cdot R > -||\hat{\boldsymbol{g}_i}|| \iff -\epsilon(R)\cdot R > -z(R) ||\boldsymbol{g}_i|| -||\hat{\boldsymbol{g}}_i||, $$
which holds if and only if \begin{equation} \label{eq:iff} \tfrac{1}{R}\left[ z(R) ||\boldsymbol{g}_i|| +||\hat{\boldsymbol{g}}_i||\right]> \epsilon(R). \end{equation}
Now observing that
$$\tfrac{1}{R}\left[ z(R) ||\boldsymbol{g}_i|| +||\hat{\boldsymbol{g}}_i||\right]\geq \tfrac{1}{R}\left[ z(R) ||\boldsymbol{g}_i|| +||\boldsymbol{g}_i||-\epsilon(R)\right]; $$
and, the right hand side is greater then $\epsilon(R)$ because
$$\tfrac{1}{R}\left[ z(R) ||\boldsymbol{g}_i|| +||\boldsymbol{g}_i||-\epsilon(R)\right]> \epsilon(R)$$ $$ \iff \tfrac{1}{R}\left[ z(R) ||\boldsymbol{g}_i|| +||\boldsymbol{g}_i||\right]>(1+1/R)\epsilon(R) $$ $$\iff \tfrac{z(R)+1}{R+1}||\boldsymbol{g}_i||>\epsilon(R);$$
and, by the definition of $\epsilon(R)$, this inequality holds.

\textbf{Part 2.}
If $\boldsymbol{v}_I$ satisfies $ \boldsymbol{g}_{i}^T \boldsymbol{v}_I < -||\boldsymbol{g}_{i}|| \text{ for all i = 1, ...,m }$, then clearly $s_i = \boldsymbol{g}_{i}^T\boldsymbol{v}_I+||\boldsymbol{g}_{i}||<0$ for all $i=1,...,m$. Now, chose any $\epsilon>0$ such that $\epsilon(1+||\boldsymbol{v}_I||)<-s_i$ for all $i = 1, ..., m$. Notice that
 $$\hat{\boldsymbol{g}}_i^T\boldsymbol{v}_I = \boldsymbol{g}_i^T\boldsymbol{v}_I+ (\hat{\boldsymbol{g}}_i-\boldsymbol{g}_i)^T\boldsymbol{v}_I = \boldsymbol{g}_i^T\boldsymbol{v}_I+||\boldsymbol{g}_i||-||\boldsymbol{g}_i||+ (\hat{\boldsymbol{g}}_i-\boldsymbol{g}_i)^T\boldsymbol{v}_I ;$$
and thus
 $$\hat{\boldsymbol{g}}_i^T\boldsymbol{v}_I \leq s_i-||\boldsymbol{g}_i||+ \epsilon ||\boldsymbol{v}_I|| \leq s_i-||\hat{\boldsymbol{g}}_i||+\epsilon+ \epsilon ||\boldsymbol{v}_I||  = s_i-||\hat{\boldsymbol{g}}_i||+\epsilon (1+ ||\boldsymbol{v}_I||);$$
and, since $\epsilon$ was chosen to satisfy $\epsilon(1+||\boldsymbol{v}_I||)<-s_i$ for all $i = 1, ..., m$, we have that
$\hat{\boldsymbol{g}}_i^T\boldsymbol{v}_I<-||\hat{\boldsymbol{g}}_i||$ for all $i=1,...,m$.

\textbf{Part 3.}
If $\boldsymbol{v}_E$ satisfies $ \boldsymbol{g}_{i}^T \boldsymbol{v}_E > -||\boldsymbol{g}_{i}|| \text{ for some i = 1, ...,m }$, then clearly for one such $i = i^*$ that satisfies this inequality we have $s_{i^*} = \boldsymbol{g}_{i^*}^T\boldsymbol{v}_E+||\boldsymbol{g}_{i^*}||>0$. Now, chose any $\epsilon>0$ such that $\epsilon(1+||\boldsymbol{v}_E||)< s_{i^*}$. Notice that
 $$\hat{\boldsymbol{g}}_{i^*}^T\boldsymbol{v}_E = \boldsymbol{g}_{i^*}^T\boldsymbol{v}_E+ (\hat{\boldsymbol{g}}_{i^*}-\boldsymbol{g}_{i^*})^T\boldsymbol{v}_E = \boldsymbol{g}_{i^*}^T\boldsymbol{v}_E+||\boldsymbol{g}_{i^*}||-||\boldsymbol{g}_{i^*}||+ (\hat{\boldsymbol{g}}_{i^*}-\boldsymbol{g}_{i^*})^T\boldsymbol{v}_E ;$$
and thus
 $$\hat{\boldsymbol{g}}_{i^*}^T\boldsymbol{v}_E \geq s_{i^*}-||\boldsymbol{g}_{i^*}||- \epsilon ||\boldsymbol{v}_E|| \geq s_{i^*}-||\hat{\boldsymbol{g}}_{i^*}||-\epsilon- \epsilon ||\boldsymbol{v}_E||  = s_{i^*}-||\hat{\boldsymbol{g}}_{i^*}||-\epsilon (1+ ||\boldsymbol{v}_E||);$$
and, since $\epsilon$ was chosen to satisfy $\epsilon(1+||\boldsymbol{v}_E||)<s_{i^*}$  we have that
$\hat{\boldsymbol{g}}_{i^*}^T\boldsymbol{v}_E>-||\hat{\boldsymbol{g}}_{i^*}||$.
\end{proof}

\end{document}